\documentclass[a4paper,11pt,reqno]{amsart}
\usepackage[utf8]{inputenc}
\usepackage{amsmath,amssymb,geometry,hyperref,graphicx}

\theoremstyle{plain}
\newtheorem{theorem}{Theorem}[section]
\newtheorem{corollary}[theorem]{Corollary}

\title[Lacunary Recurrences for Lucas Numbers]{A family of lacunary recurrences for Lucas Numbers}
	\author[P. J. Mahanta]{Pankaj Jyoti Mahanta}
	\address{Gonit Sora, Dhalpur, Assam 784165, India} 
	\email{pankaj@gonitsora.com}
	\author[M. P. Saikia]{Manjil P. Saikia}
	\address{School of Mathematics, Cardiff University, Cardiff CF24 4AG, UK} 
	\email{SaikiaM@cardiff.ac.uk, manjil@gonitsora.com}
	\thanks{The second author is partially supported by the Leverhulme Trust Research Project Grant RPG-2019-083.}

\date{\today}

\begin{document}

\begin{abstract}
We prove an infinite family of lacunary recurrences for the Lucas numbers using combinatorial means.
\end{abstract}

\maketitle

\section{Introduction}

A recurrence relation involving only terms of a given sequence with indices in arithmetic progression is called a \textit{lacunary recurrence}. The \textit{gap} of such a lacunary recurrence is the common difference in the indices in arithmetic progression. Several such lacunary recurrences are known for sequences including but not limited to Bernoulli numbers, Euler numbers, $k$-Fibonacci numbers, etc. We refer the reader to the recent paper of Ballantine and Merca \cite{merca} for relevant references and other examples.

Ballantine and Merca \cite{merca} proved an infinite family of lacunary recurrences for Fibonacci numbers. They closed the paper by asking the natural question of whether such an infinite family of lacunary recurrences can be found for the Lucas numbers. The aim of this article is to prove such an infinite family of lacunary recurrences. Before stating and proving our result, let us recall some definitions and relations. 

The Fibonacci sequence $\{F_n\}_{n\geq 0}$ defined by the recurrence relation
\[
F_n=F_{n-1}+F_{n-2},
\]
with $F_0=0$ and $F_1=1$. We use the convention $F_n=0$ when $n<0$. Similarly, the Lucas sequence $\{L_n\}_{n\geq 0}$ defined by the recurrence relation
\[
L_n=L_{n-1}+L_{n-2},
\]
with $L_0=2$ and $L_1=1$. We use the convention $L_n=0$ when $n<0$. These two sequences are related by the identity
\begin{equation}\label{eq:lucas}
    L_n=F_{n-1}+F_{n+1}.
\end{equation}

Several interesting relationships between Fibonacci numbers are known, two of them, which are relevant for this paper are
\begin{equation}\label{eq:1}
    F_{m+n}=F_{m}F_{n+1}+F_{m-1}F_{n},
\end{equation}
and
\begin{equation}\label{eq:2}
    (-1)^nF_{m-n}=F_{m}F_{n+1}-F_{m+1}F_n.
\end{equation}
For these and many other identities we refer the reader to Honsberger's book \cite[Chapter 8]{hons} and to the more recent book by Koshy \cite[Chapter 5]{koshy}. The second identity \eqref{eq:2} is called d'Ocagne's identity.

Lucas in 1876 proved a lacunary recurrence of gap $2$ for the Fibonacci numbers in the following equivalent form
\[
F_n=\frac{1+(-1)^n}{2}+F_{n-2}+\sum_{k=1}^{\lfloor \frac{n-1}{2}\rfloor}F_{n-2k}.
\]
This was generalized by Ballantine and Merca \cite{merca} to the following
\begin{theorem}\cite[Theorem 1]{merca}\label{thm:merca}
Given a positive integer $N\geq 2$, we have
\[
F_n=F_N\cdot F_{N-1}^{\lfloor \frac{n-1}{N}\rfloor +1}\cdot F_{(n-1)~\text{mod}~N}+F_{N+1}\cdot F_{n-N}+F_N^2\cdot \sum_{k=2}^{\lfloor {\frac{n-1}{N} \rfloor}}F_{N-1}^{k-2}\cdot F_{n-kN},
\]
for all $n\geq N$.
\end{theorem}
\noindent This result is also valid for Pell numbers (as remarked in Section \ref{sec:con}).

It is quite natural to ask, as did Ballantine and Merca \cite{merca} if a similar result holds for the Lucas numbers? We now present such a result in the following theorem.
\begin{theorem}\label{thm:main}
Given a positive integer $N\geq 2$, we have
\begin{equation}\label{eq:main}
    L_n=L_N\sum_{i=1}^{d}{(-1)}^{(N+1)(i+1)}L_{n-(2i-1)N}+{(-1)}^{(N+1)(d+2)}L_{n-2dN},
\end{equation}
where $d=\left\lfloor \left\lfloor \dfrac{n}{N}\right\rfloor /2\right\rfloor$ and $\dfrac{n}{2}\geq N\geq 0$.
\end{theorem}

A simple consequence of the above theorem is the following congruence.
\begin{corollary}
For a given integer $N\geq 2$ we have
\[
L_n-{(-1)}^{(N+1)(d+2)}L_{n-2dN} \equiv 0 ~ (\textup{mod}~ L_N),
\]
where $d=\left\lfloor \left\lfloor \dfrac{n}{N}\right\rfloor /2\right\rfloor$ and $\dfrac{n}{2}\geq N\geq 0$.
\end{corollary}

\section{A Combinatorial Proof of Theorem \ref{thm:main}}

It is well-known that the Fibonacci numbers can be interpreted as tilings of an $1 \times n$ board with squares and dominoes. We call such a board an $n$-board. If the number of such tilings is $f_n$, then it can be proved that $F_{n+1}=f_n$ (see for instance, the book by Benjamin and Quinn \cite{proofs}). With this notation, equations \eqref{eq:1} and \eqref{eq:2} now becomes
\begin{equation}\label{eq:p}
    f_mf_n+f_{m-1}f_{n-1}=f_{m+n}
\end{equation}
and
\begin{equation}\label{eq:q}
    f_{m-1}f_n-f_mf_{n-1}={(-1)}^{n}f_{m-n-1}
\end{equation}
Both these identities can be easily proven using the combinatorial interpretation of $f_k$.

It is also known (see Chapter 2 of the book by Benjamin and Quinn \cite{proofs}) that the number $l_n$ of ways to tile a circular board composed of $n$ labelled cells with curved squares and dominoes is equal to $L_n$. We call such a tiling of the circular $n$-board to be an $n$-bracelet. There are two types of bracelets, an \textit{in-phase} or an \textit{out-of-phase}. A bracelet is out-of-phase if a domino covers the cells numbered $n$ and $1$, and it is called in-phase if it is not out-of-phase. An example of an out-of-phase $4$-bracelet and an in-phase $4$-bracelet is shown in Figure \ref{fig:ex}, where dominoes are coloured black and squares are white. We note that an in-phase tiling of an $n$-bracelet can be made into a tiling of an $n$-board. From this observation it is easy to see the validity of equation \eqref{eq:lucas}. We are now in a position to prove Theorem \ref{thm:main}.
\begin{figure}
\scalebox{0.9}{\includegraphics{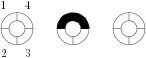}}
\caption{\label{fig:ex} Examples of bracelets.}
\end{figure}

\begin{proof}[Proof of Theorem \ref{thm:main}]
Let us draw two sets of circular boards as shown in Figure \ref{fig:1}, and call them Set 1 and Set 2. We mark the cells as
shown in Figure \ref{fig:1}. The number of bracelets of Set-1 is $L_n$ and of Set-2 is
$L_N\times L_{n-N}$, where $\frac{n}{2}\geq N\geq 0.$

\begin{figure}
\scalebox{0.14}{\includegraphics{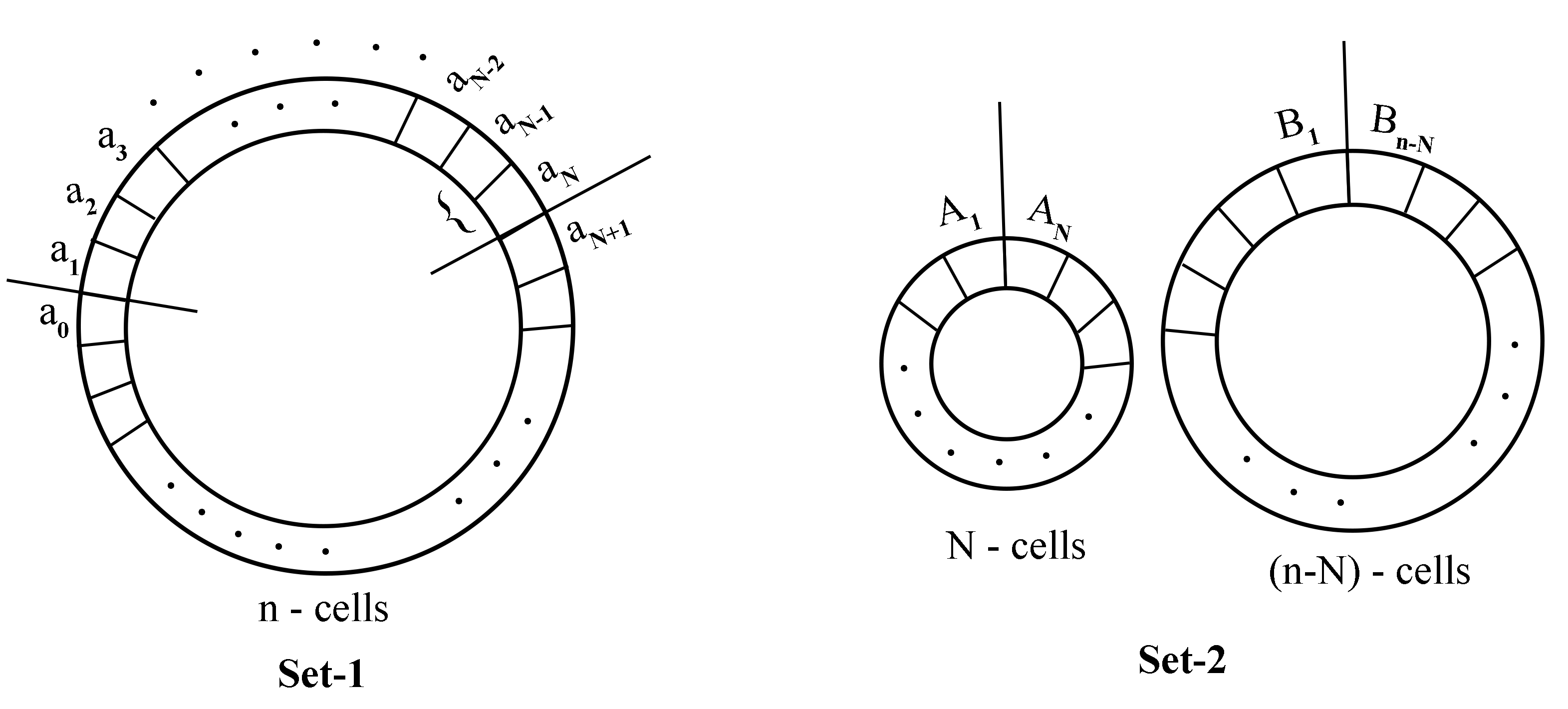}}
\caption{\label{fig:1} The two sets of bracelets considered in the proof of Theorem \ref{thm:main}.}
\end{figure}

We can break the tilings of Set-2 in the following four
parts:
\begin{itemize}
    \item[(a)] $f_N\times f_{n-N}.$ (Here both the $N$-bracelet and $(n-N)$-bracelets are in-phase.)
\item[(b)] $f_{N-2}\times f_{n-N}.$ (Here only the $N$-bracelet is
out-of-phase.)
\item[(c)] $f_N\times f_{n-N-2}.$ (Here only the $(n-N)$-bracelet is
out-of-phase.)
\item[(d)] $f_{N-2}\times f_{n-N-2}.$ (Here both are out-of-phase.)
\end{itemize}

Observe that the tilings of (a) can be made into tilings of the $n$-bracelet in such a way that the $N$-board covers the cells of the $n$-bracelet from
$a_1$ to $a_N.$ And hence the $(n-N)$-board covers the remaining
cells of the $n$-bracelet. In these tilings of the $n$-bracelet, there is no domino which covers the cells $a_0$ and $a_1$ or $a_N$ and $a_{N+1}$.

Observe that the tilings of (b) can be made into tilings of the $n$-bracelet in such a way that a domino covers the cells $a_0$ and $a_1$, and the $N$-board covers the cells from $a_0$ to $a_{N-1}.$ So, the
$(n-N)$-board covers the remaining cells of the n-bracelet. In these tilings no domino covers the cells
$a_{N-1}$ and $a_N.$

There are only two types of tilings that remains in the set of all tilings of the $n$-bracelet, apart from the ones discussed above:
\begin{itemize}
    \item[(1)] Tilings where one domino covers the cells $a_0$ and $a_1$ and another domino covers the cells $a_{N-1}$ and $a_N$. The total number of such tilings is $f_{N-3}f_{n-N-1}$.
\item[(2)] Tilings where one domino covers the cells $a_{N}$ and $a_{N+1}$, but no domino covers the cells $a_0$ and $a_1$. The total number of such tilings is $f_{n-2}-f_{N-2}f_{n-N-2}.$
\end{itemize}

Let us now compute the difference (say $A$) between the total tilings of (1) and (2) and that of the total tilings in (c) and (d)
\begin{multline*}
    A:=(f_{N-3}f_{n-N-1}+f_{n-2}-f_{N-2}f_{n-N-2})-(f_Nf_{n-N-2}+f_{N-2}f_{n-N-2})\\
    =f_{N-3}f_{n-N-1}+f_{(N-1)+(n-N-1)}-2f_{N-2}f_{n-N-2}-f_Nf_{n-N-2}.
\end{multline*}
Using equation \eqref{eq:p} above we get
\begin{align*}
     A=& f_{N-3}f_{n-N-1}+f_{N-1}f_{n-N-1}+f_{N-2}f_{n-N-2}-2f_{N-2}f_{n-N-2}-f_Nf_{n-N-2}\\
     =& -(f_{n-N-2}f_{N-2}-f_{n-N-1}f_{N-3})-(f_{n-N-2}f_N-f_{n-N-1}f_{N-1}).
\end{align*}
Using equation \eqref{eq:q} above we get
\begin{align*}
     A=& -{(-1)}^{N-2}f_{(n-N-1)-(N-2)-1}-{(-1)}^{N}f_{(n-N-1)-N-1}\\
     =& {(-1)}^{N-1}f_{n-2N}+{(-1)}^{N+1}f_{n-2N-2}\\
     =& {(-1)}^{N+1}L_{n-2N}.
\end{align*}
In the last step we used equation \eqref{eq:lucas}. Hence
\[
f_{N-3}f_{n-N-1}+f_{n-2}-f_{N-2}f_{n-N-2}=f_Nf_{n-N-2}+f_{N-2}f_{n-N-2}+{(-1)}^{N+1}L_{n-2N}.
\]

Finally, adding the total number of the other tilings (namely those in (a) and (b)) to both sides of the above we get
\begin{equation}\label{eq:a}
    L_n=L_N\times L_{n-N}+{(-1)}^{N+1}L_{n-2N}.
\end{equation}
The left hand side follows because the number of tilings in (1), (2), (a) and (b) is $L_n$, while the right hand side follows because the number of tilings in (a)--(d) is $L_N\times L_{n-N}$. Replacing $n$ by $n-2N$ in \eqref{eq:a}, we get
\begin{equation}\label{eq:b}
    L_{n-2N}=L_NL_{n-3N}+{(-1)}^{N+1}L_{n-4N}.
\end{equation}
Therefore, from equations \eqref{eq:a} and \eqref{eq:b}, we get
\[
L_n=L_NL_{n-N}+{(-1)}^{N+1}L_NL_{n-3N}+L_{n-4N}.
\]
Again,
\[
L_{n-4N}=L_NL_{n-5N}+{(-1)}^{N+1}L_{n-6N}.
\]
So, we have
\begin{align*}
    L_n =& L_NL_{n-N}+{(-1)}^{N+1}L_NL_{n-3N}+L_NL_{n-5N}+{(-1)}^{N+1}L_{n-6N}\\
    =& L_NL_{n-N}+{(-1)}^{N+1}L_NL_{n-3N}+L_NL_{n-5N}+{(-1)}^{N+1}L_NL_{n-7N}+L_{n-8N}\\
    =& L_N\left(L_{n-N}+{(-1)}^{N+1}L_{n-3N}+L_{n-5N}+{(-1)}^{N+1}L_{n-7N}\right)+L_{n-8N}.
\end{align*}
This gives us,
\begin{multline*}
    L_n=L_N\left({(-1)}^{(N+1)(1+1)}L_{n-N}+{(-1)}^{(N+1)(2+1)}L_{n-3N}\right.\\ \left.+{(-1)}^{(N+1)(3+1)}L_{n-5N}+{(-1)}^{(N+1)(4+1)}L_{n-7N}\right)+{(-1)}^{(N+1)(5+1)}L_{n-8N}.
\end{multline*}
We can proceed in this way up to the $\left\lfloor \left\lfloor \dfrac{n}{N}\right\rfloor /2\right\rfloor$-th step. This proves equation \eqref{eq:main}.
\end{proof}

\section{Concluding Remarks}\label{sec:con}
We can combine several known identities involving Lucas and Fibonacci numbers with Theorems \ref{thm:merca} and \ref{thm:main} to give several new results involving more complicated sums. We do not explore this here.

A generalization of the Fibonacci sequence, called the Gibonacci sequence $\{G_n\}_{n\geq 0}$ is given by the same recurrence
\[
G_n=G_{n-1}+G_{n-2},
\]
for all $n\geq 2$. Changing the initial conditions for $G_0$ and $G_1$ gives rise to different sequences, two of which are the Fibonacci and Lucas sequences. There exist combinatorial interpretations for such a Gibonacci sequence, which are similar to the interpretation for the Lucas sequence. It would seem that by tweaking our proofs, a more general lacunary recurrence could be found for the Gibonacci sequence. We leave this as an open problem.

Another remark is that, Theorem \ref{thm:merca} is actually valid also for Pell numbers. The sequence of Pell numbers $\{P_n\}_{n\geq 0}$ is given by the recurrence
\[
P_{n}=2P_{n-1}+P_{n-2},
\]
with $P_0=1$ and $P_1=1$. This can be seen from the combinatorial interpretation of Pell numbers given by Benjamin, Plott and Sellers \cite{pell}, and combining it with the proof of Ballantine and Merca \cite{merca} where the proof is independent of whether we use the combinatorial interpretation of the Fibonacci numbers or the Pell numbers. Thus, we have the following result.
\begin{theorem}
Given a positive integer $N\geq 2$, we have
\[
P_n=P_N\cdot P_{N-1}^{\lfloor \frac{n-1}{N}\rfloor +1}\cdot P_{(n-1)~\text{mod}~N}+P_{N+1}\cdot P_{n-N}+P_N^2\cdot \sum_{k=2}^{\lfloor {\frac{n-1}{N} \rfloor}}P_{N-1}^{k-2}\cdot P_{n-kN},
\]
for all $n\geq N$.
\end{theorem}

\section*{Acknowledgements}
The authors are grateful to the anonymous referee for her/his helpful comments.

\bibliographystyle{alpha}

\end{document}